\documentclass[twoside, 12pt]{amsart}
\usepackage[letterpaper,hmargin=1.25in,vmargin=1.5in]{geometry}
\usepackage{amscd}
\usepackage{verbatim}
\usepackage{color}
\usepackage{comment}

\input xy
\xyoption{all}

\usepackage{amssymb}

\DeclareMathAlphabet{\cat}{OT1}{cmss}{m}{sl}

\newtheorem{theorem}{Theorem}[section]

\newtheorem{proposition}[theorem]{Proposition}
\newtheorem{lemma}[theorem]{Lemma}

\theoremstyle{definition}
\newtheorem{remark}[theorem]{Remark}

\newtheorem{example}[theorem]{Example}

\newcommand{\tens}{\otimes}

\newcommand{\ind}{\operatorname{\hspace{0.3mm}ind}}

\newcommand{\Br}{\operatorname{Br}}

\newcommand{\Gal}{\operatorname{Gal}}

\newcommand{\gPGL}{\operatorname{\mathbf{PGL}}}

\newcommand{\ed}{\operatorname{ed}}

\newcommand{\td}{\operatorname{tr.deg}}

\newcommand{\torsor}{\operatorname{\cat{torsors}}}

\newcommand{\cT}{\mathcal T}

\title[A lower bound on the essential dimension of $\gPGL_{4}$ in characteristic $2$] 
{A lower bound on the essential dimension of $\gPGL_{4}$ in characteristic $2$}

\author
[S. Baek] {Sanghoon Baek}

\address
{Department of Mathematical Sciences, KAIST, 291 Daehak-ro, Yuseong-gu, Daejeon, 305-701, Republic of Korea}

\email {sanghoonbaek@kaist.ac.kr}

\begin{document}

\maketitle

\begin{abstract}
In the present paper, we provide a lower bound of the essential dimension over a field of positive characteristic via Kato's cohomology group, defined by cokernel of a general Artin-Schreier operator. Combining this with Tignol's result on the second trace form of simple algebras of degree $4$, we show that $\ed(\gPGL_{4})\geq 4$ over a field of characteristic $2$.
\end{abstract}


\section{Introduction}\label{intro}

Let $F$ be a field and $\cT : \cat{Fields}/F\to\cat{Sets}$ be a functor from the category $\cat{Fields}/F$ of field extensions over $F$ to the category $\cat{Sets}$ of sets. We shall write $\ed(\cT)$ for the essential dimension of $\cT$.

A way to find lower bounds for the essential dimension of $\cT$ is to give \emph{nontrivial} cohomological invariants associated to $\cT$. More precisely, if there exists a nontrivial cohomological invariant of degree $n$ with values in a torsion Galois module of exponent not divisible by the characteristic of $F$, then $\ed(\cT)\geq n$. This approach produces lower bounds for various algebraic structures \cite[\S12]{Reich}, but may not work when the exponent is a multiple of the characteristic of $F$.

In this paper, we presents a lower bound of the essential dimension based on the cohomological invariants with values in Kato's cohomology group. As an application, we compute the essential dimension of the functor given by the classes of $n$-fold quadratic Pfister forms.

In the last section, we provide our main result, $\ed(\gPGL_{4})\geq 4$ over a field of characteristic $2$ by relating the result in the previous section and Tignol's theorem on the second trace form of simple algebras of degree $4$.

\section{Lower bound via cohomological invariants}

Let $F$ be field of characteristic $p>0$. For any $n\geq 0$, we denote by $\Omega_{F}^{n}$ the $n$th exterior power of the absolute differential module $\Omega_{F}^{1}$. Let $d: \Omega_{F}^{n}\to \Omega_{F}^{n+1}$ be the exterior derivation given by $d(ydx_{1}\wedge \ldots \wedge dx_{n})=dy\wedge dx_{1}\wedge \ldots \wedge dx_{n}$.

We write $H_{p}^{n+1}(F)$ for the cokernel of the homomorphism $\wp: \Omega_{F}^{n}\to \Omega_{F}^{n}/d\Omega_{F}^{n-1}$ defined by 
\[\wp(y\frac{dx_{1}}{x_{1}}\wedge \ldots \wedge \frac{dx_{n}}{x_{n}})=(y^{p}-y)\frac{dx_{1}}{x_{1}}\wedge \ldots \wedge \frac{dx_{n}}{x_{n}}.\]

For a field extension $K/F$, we have a \emph{cohomological functor of degree $n+1$}
\[H:\cat{Fields}/F\to \cat{Abelian\, Groups}\,\, \text{ by } K\mapsto H^{n+1}_{p}(K).\]
A \emph{cohomological invariant} of degree $n+1$ associated to a functor $\cT:\cat{Fields}/F\to \cat{Sets}$ is a morphism of functors $q:\cT\to H$. An invariant $q$ is \emph{nontrivial} if there exists a field extension $K/F$ containing an algebraic closure $\bar{F}$ of $F$ and an element $t\in \cT(K)$ such that $q(K)(t)\neq 0$ in $H(K)$.

\begin{lemma}\label{basiclemma}
Let $F$ be a field of characteristic $p>0$ and let $\cT: \cat{Fields}/F\to \cat{Sets}$ be a functor. If there exists a nontrivial invariant of degree $n+1$ associated to $\cT$, then $ed(\cT)\geq n+1$.
\end{lemma}
\begin{proof}
Since essential dimension goes down with field extensions, we may replace $F$ by $\bar{F}$. Let $K/F$ be a field extension such that there exists $t\in \cT(K)$ with $q(K)(t)\neq 0$ in $H(K)$. Assume that $\ed(\cT)\leq n$. Then, there exist an intermediate field $E$ of $K/F$ and $t'\in \cT(E)$ such that $\td_{F}E\leq n$ and $t=t'_{K}$. Therefore, by \cite[Chap. II, \S 4, Exercises 3)(a)]{Serre} (see also \cite[Proposition]{AB}), we have $H(E)=0$. This leads to a contradiction by the following commutative diagram
\[
\begin{CD}
\cT(E)  @>q(E)>> H(E)     \\
@V{}VV    @VV{} V     \\
\cT(K)  @>q(K)>> H(K)     
\end{CD}
\]
\end{proof}

\begin{example}
Let $F$ be a field of characteristic $2$. Consider the functor $P_{n}:\cat{Fields}/F\to \cat{Sets}$ which assigns to a field extension $K/F$ the set of $K$-isomorphism classes of $n$-fold quadratic Pfister forms 
\[\langle\langle a_{1}, \ldots, a_{n-1}, a_{n}]]:=\langle 1, a_{1}\rangle\tens \cdots \tens\langle 1, a_{n-1}\rangle\tens [1, a_{n}].\]
Consider the cohomological invariant $q:P_{n}\to H$ of degree $n$ defined by
\[\langle\langle a_{1}, \ldots, a_{n-1}, a_{n}]]\mapsto a_{n}\frac{da_{1}}{a_{1}}\wedge \ldots \wedge \frac{da_{n-1}}{a_{n-1}}.\] We show that $q$ is nontrivial. Let $L=\bar{F}(x_{1}, \ldots, x_{n})$ with independent variables $x_{1},\ldots, x_{n}$. Then, the form $\phi:=\langle\langle x_{1}, \ldots, x_{n-1}, x_{n}]]$ is anisotropic over $L$. Hence, by \cite[Theorem (2)]{Kato2}, $q(L)(\phi)$ is nonzero in $H(L)$. By Lemma \ref{basiclemma}, we have $\ed(P_{n})\geq n$. Thus, $\ed(P_{n})=n$.

\end{example}

\begin{example}
As $G_{2}$-$\torsor$ is bijective to $P_{3}$, we have $\ed(G_{2})=3$ over a field of characteristic $2$.
\end{example}

\section{A lower bound of $\ed(\gPGL_{4})$ in characteristic $2$}

Let $F$ be a field of characteristic $2$. For any central simple $F$-algebra $A$ of degree $n$, there exists a quadratic form $\phi_{A}:A\to F$ given by $\phi_{A}(a)=s_{2}(a)$, the second coefficient of the reduced characteristic polynomial of $A$ \[\operatorname{Prd}_{A,a}(x)=x^{n}-s_{1}(a)x^{n-1}+s_{2}(a)x^{n-2}-\cdots +(-1)^{n}s_{n}(a).\]

In \cite[Th\'eor\`eme 1]{Tignol} Tignol showed that for any central simple $F$-algebra $A$ of degree $4$ there exist $2$-fold Pfister form $\phi_{2, A}$ and $4$-fold Pfister form $\phi_{4, A}$ such that
\begin{equation}\label{q2q4}
\phi_{A}=[1, 1]+ \phi_{2, A}+ \phi_{4, A}\in W_{q}(F).
\end{equation}
Moreover, $\phi_{2, A}$ is a norm form of $Q$ such that $\big[A\tens A\big]=\big[Q\big]$ in $\Br(F)$ (i.e., $\phi_{2, A}=\phi_{Q}$) and $\phi_{4, A}=\langle\langle a_{1}, a_{2}\rangle\rangle\tens \phi_{2, A}$ for some $a_{1}, a_{2}\in F^{\times}$.

By a theorem of Kato \cite[Theorem (2)]{Kato2}, we have an isomorphism
\[e_{n+1}: I^{n+1}_{q}(K)/I^{n+2}_{q}(K)\to H^{n+1}_{2}(K)\]
for any field extension $K/F$. Therefore, the quadratic form $\phi_{2, A}$ (resp. $\phi_{4, A}$) in (\ref{q2q4}) defines a cohomological invariant of degree $2$ (resp. $4$) associated to $\gPGL_{4}$-$\torsor$:
\begin{equation}\label{pglfour}
q: \gPGL_{4}\!\text{-}\!\torsor\to H \text{ given by }A\mapsto e_{2}(\phi_{2, A}) \text{ (resp. }e_{4}(\phi_{4, A})).
\end{equation}
 
We recall parametrizations of central simple $F$-algebras of degree $2$ and $4$. For a field extension $K/F$, $a\in K^{\times}$, and $b\in K$, we write $Q:=\big(a, b\big]_{K}$ for the quaternion algebra $K\oplus Ki\oplus Kj\oplus Kji$ with the multiplication $i^{2}=a$, $j^{2}+j=b$, and $ji=ij+i$. By a simple calculation, we have $\phi_{Q}=\langle\langle a, b]]$.

For $a\in K$, we write $K_{a}$ for the quadratic \'etale $K$-algebra $K[t]/(t^{2}+t+a)$. Let $L$ be a biquadratic \'etale $K$-algebra which is generated by two elements $\alpha$ and $\beta$ subject to the relations $\alpha^{2}+\alpha=a$, $\beta^{2}+\beta=b$, $\alpha\beta=\beta\alpha$ for some $a, b\in K$. We shall denote such an algebra by $K_{a, b}$. The automorphism group of $K_{a, b}/K$ is generated the automorphisms $\sigma$ and $\tau$ with the following relations: \[\sigma(\alpha)=\alpha,\, \sigma(\beta)=\beta+1,\, \tau(\beta)=\beta,\, \tau(\alpha)=\alpha+1.\]

Let $u\in K_{a}^{\times}$, $v\in K_{b}^{\times}$, and $w\in K_{a+b}^{\times}$ satisfying $N_{K_{a}/K}(u)\cdot N_{K_{b}/K}(v)=N_{K_{a+b}/K}(w)$. Consider the algebra $K_{a, b}\oplus K_{a, b}\cdot\lambda\oplus K_{a, b}\cdot\delta\oplus K_{a, b}\cdot\lambda\delta$ over $K$ with the multiplication 
\[\lambda^{2}=u,\, \delta^{2}=v,\, (\lambda\delta)^{2}=w,\, \lambda c=\sigma(c)\lambda,\, \delta c =\tau(c)\delta\]
for all $c\in K_{a, b}$. We write $(K_{a, b}, u, v, w)$ for such an algebra. By a theorem of Albert, every central simple $K$-algebra $A$ of degree $4$ is isomorphic to for some $(K_{a, b}, u, v, w)$.

\begin{proposition}
Let $F$ be a field of characteristic $2$. Then, $\ed(\gPGL_{4})\geq 4$.
\end{proposition}

\begin{proof}
We first show that the cohomological invariant of degree $4$ in (\ref{pglfour}) is nontrivial. To do this, we slightly modify a generic abelian crossed product algebra given in \cite[4. Exemple]{Tignol} and \cite{AS}. Let $K=\bar{F}(x, y)$ with independent variables $x, y$. Consider a division algebra $A:=(K_{x, y}, u, v, w)$ over $K$ such that $\ind(A)=\exp(A)=4$ for some $u\in K_{x}^{\times}, v\in K_{y}^{\times}, w\in K_{x, y}^{\times}$. By a simple modification, we may assume that $u=f+\alpha$ and $v=g+\beta$ with nonzero $f, g\in K$.

Let $s$ and  $t$ be independent variables over $K$. Extend the automorphisms $\sigma$ and $\tau$ of $\Gal(K_{x, y}/K)$ to the automorphisms of $\Gal(K_{x, y}(s, t)/K(s, t))$ by fixing $s$ and $t$. Let $B:=(K_{x, y}(s, t), s, t, st)=\big(s, x\big]_{K(s, t)}\tens \big(t, y\big]_{K(s, t)}$ be a biquaternion algebra over $K(s, t)$. Then, $A_{K(s, t)}\tens B$ is Brauer equivalent to a division $K(s, t)$-algebra $D:=(K_{x, y}(s, t), us, vt, wst)$ such that $\ind(D)=\exp(D)=4$. In $\Br(K(s, t))$, we have
\[\big[D\tens D\big]=\big[A_{K(s, t)}\tens A_{K(s, t)}\big]=\big[\big(N_{K_{y}/K}(v), x\big]_{K(s, t)}\big],\]
thus, we obtain an anisotropic form $\phi_{2, D}=\langle\langle N_{K_{y}/K}(v), x]]=\langle\langle g^{2}+g+y, x]]$ over $K(s, t)$ in the decomposition of $\phi_{D}$.

Since $\phi_{D}(\lambda)=s$ and $\phi_{D}(\delta)=t$, it follows from \cite[Lemme]{Tignol} and \cite[Theorem 2.10]{KKOPS} that $\phi_{4, D}=\langle\langle s, t, g^{2}+g+y, x]]$. As $s$ and $t$ are independent variables over $K$ and $\phi_{2, D}$ is anisotropic over $\bar{F}(x, y, s, t)$, the form $\phi_{4, D}$ is anisotropic over $\bar{F}(x, y, s, t)$. Therefore, by \cite[Theorem (2)]{Kato2}, the invariant $q(\bar{F}(x, y, s, t))(\phi_{4, D})$ in (\ref{pglfour}) is nonzero in $H(\bar{F}(x, y, s, t))$. Finally, by Lemma \ref{basiclemma} we obtain $\ed(\gPGL_{4})\geq 4$.\end{proof}

\begin{remark}
As $\ed(\gPGL_{4})\leq 5$, we have $\ed(\gPGL_{4})$ is either $4$ or $5$. When $\operatorname{char}(F)\neq 2$, it was shown by Rost that $\ed(\gPGL_{4})=5$.
\end{remark}

\paragraph{\bf Acknowledgments.} 
This work was partially supported by an internal fund from KAIST, TJ Park Junior Faculty Fellowship of POSCO TJ Park Foundation, and National Research Foundation of Korea (NRF) funded by the Ministry of Science, ICT and Future Planning (2013R1A1A1010171).

\end{document}